\documentclass[12pt]{amsart}

\usepackage{amsfonts,amssymb,stmaryrd,amscd,amsmath,latexsym,amsbsy}

\newtheorem{theorem}{Theorem}[section]

\newtheorem{corollary}[theorem]{Corollary}
\theoremstyle{definition}
\newtheorem{definition}[theorem]{Definition}

\newtheorem{conjecture}[theorem]{Conjecture}
\newtheorem{remark}[theorem]{Remark}

\newcommand{\h}{\mathfrak{h}}

\newcommand{\ben}{\begin{enumerate}}
\newcommand{\een}{\end{enumerate}}

\theoremstyle{plain}

\newtheorem*{sol}{Solution}

\theoremstyle{definition}

\theoremstyle{remark}

\newcommand{\solu}[1]{\begin{sol}{\bf (\ref{#1})}}

\def\h{\mathfrak{h}}

\begin{document}

\title{Proof of the Brou\'e-Malle-Rouquier conjecture in characteristic zero\\ (after I. Losev and I. Marin --- G. Pfeiffer)}

\author{Pavel Etingof}
\address{Department of Mathematics, Massachusetts Institute of Technology,
Cambridge, MA 02139, USA}
\email{etingof@math.mit.edu}

\maketitle

The goal of this note is to explain a proof of the Brou\'e-Malle-Rouquier conjecture (\cite{BMR}, p.178), stating 
that the Hecke algebra of a finite complex reflection group $W$ is free of rank $|W|$ over the algebra of parameters, 
over a field of characteristic zero. This result is not original --- it follows immediately from the results of \cite{L}, \cite{MP}, and \cite{ER}, 
but it does not seem to have been stated explicitly in the literature, so we state and prove it for future reference. 

We note that there have been a lot of results on this conjecture for particular complex reflection groups, reviewed in \cite{M3}, e.g. \cite{A,AK,M1,M2}; we are not giving the full list of references here. 

\section{The main result} 

Let $V$ be a finite dimensional complex vector space, and $W\subset GL(V)$ a finite complex reflection group, i.e., $W$ is generated by 
complex reflections (elements $s$ such that ${\rm rank}(1-s)=1$). Let $S\subset W$ be the set of reflections, and 
$V_{\rm reg}:=V\setminus \cup_{s\in S}V^s$. Then by Steinberg's theorem, $W$ acts freely on $V_{\rm reg}$. 
Let $x\in V_{\rm reg}/W$ be a base point. The {\it braid group} of $W$ is the group $B_W:=\pi_1(V_{\rm reg}/W,x)$. 
We have a surjective homomorphism $\pi: B_W\to W$ (corresponding to gluing back the reflection hyperplanes $V^s$), 
and ${\rm Ker}\pi$ is called the {\it pure braid group} of $W$, denoted by $PB_W$. For each $s\in S$, let $T_s\in B_W$
be a path homotopic to a small circle around $V^s$ (it is defined uniquely up to conjugation). Also let $n_s$ be the order of $s$. 
Then $T_s^{n_s}\in PB_W$, and by the Seifert-van Kampen theorem, $PB_W$ is the normal closure of the subgroup of $B_W$ 
generated by $T_s^{n_s}$, $s\in S$. In other words, $W$ is the quotient of $B_W$ by the relations 
$T_s^{n_s}=1$, $s\in S$. 

Let $u_{s,i}$, $i=1,...,n_s$, be variables such that $u_{s,i}=u_{t,i}$ if $s$ is conjugate to $t$ in $W$. 
Let $R:=\Bbb Z[u_{s,i}^{\pm 1}, s\in S, i\in [1,n_s]]$. 

\begin{definition} (\cite{BMR}) The Hecke algebra $H(W)$ is the quotient of the group algebra $RB_W$
by the relations 
$$
\prod_{i=1}^{n_s}(T_s-u_{s,i})=0,\ s\in S.
$$
\end{definition} 

\begin{conjecture}(\cite{BMR}, p.178)  $H(W)$ is a free $R$-module of rank $|W|$.  
\end{conjecture} 

This conjecture is currently known for all irreducible complex reflection groups except $G_{17},...,G_{21}$ (according to the Shephard-Todd classification), and there is a hope that these cases can be proved as well using a sufficiently powerful computer (see \cite{Cha1,Cha2} and \cite{M3} for more details). Also, it is shown in \cite{BMR} that to prove the conjecture, it suffices to show that $H(W)$ is spanned by $|W|$ elements.  

Our main result is 

\begin{theorem}\label{main} 
If $K$ is a field of characteristic zero then $K\otimes_{\Bbb Z} H(W)$ is a free module over $K\otimes_{\Bbb Z} R$ of rank $|W|$. 
In particular, if $q: R\to K$ is a homomorphism, then the specialization $H_q(W):=K\otimes_R H(W)$ is a $|W|$-dimensional $K$-algebra. 
\end{theorem} 

\begin{remark}\label{appl}
Theorem \ref{main} is useful in many situations, for instance in the representation theory of rational Cherednik algebras, 
where a number of results were conditional on its validity for $W$; see e.g. \cite{GGOR}, 5.4, or \cite{Sh}, Section 2. 
Also,  Theorem \ref{main} implies a positive answer to a question by Deligne and Mostow (\cite{DM}, (17.20), Question 3), which served as one of the motivations in \cite{BMR} (see \cite{BMR}, p.127). 
\end{remark}

\section{Proof of Theorem \ref{main}}

First assume that $K=\Bbb C$. 
It also suffices to assume that $W$ is irreducible. In this case, possible groups $W$ are classified by Shephard and Todd, \cite{ST}. 
Namely, $W$ belongs to an infinite series, or $W$ is one of the exceptional groups $G_n$, $4\le n\le 37$. Among these, $G_n$ with $4\le n\le 22$ are rank $2$ groups, while $G_n$ for $n\ge 23$ are of rank $\ge 3$. 

The case of the infinite series of groups is well known, see \cite{A,AK,BMR}. So it suffices to focus on the exceptional groups. 
Among these, the result is well known for Coxeter groups, which are $G_{23}=H_3$, $G_{28}=F_4$, $G_{30}=H_4$, 
$G_{35}=E_6$, $G_{36}=E_7$, $G_{37}=E_8$. 

For the groups $G_n$ for $n=24,25,26,27,29,31,32,33,34$, the result was established in \cite{MP} and references therein, see 
\cite{M3}, Subsection 4.1. Thus, Theorem \ref{main} is known (in fact, over any coefficient ring) for all $W$ except those of rank $2$. 

In the rank 2 case, the following weak version of Theorem \ref{main} was established.

\begin{theorem}\label{weak} (\cite{ER}, Theorem 6.1) If $W=G_n$, $4\le n\le 22$, then $\Bbb C\otimes_{\Bbb Z}H(W)$ is a finitely generated $\Bbb C\otimes_{\Bbb Z}R$-module. In particular, every specialization $H_q(W)$ is finite dimensional. 
\end{theorem} 

Theorem \ref{main} now follows from Theorem \ref{weak} and  the following theorem due to I. Losev. 

\begin{theorem}\label{Los} (\cite{L}, Theorem 1.1) For any $W$ and any $q: R\to \Bbb C$,  
there is a minimal two-sided ideal $I\subset H_q(W)$ such that 
$H_q(W)/I$
is finite dimensional. Moreover, we have
${\rm dim}H_q(W)/I=|W|$. 
\end{theorem} 

Namely, Theorem \ref{weak} and Theorem \ref{Los} imply that for any character $q: R\to \Bbb C$, the specialization 
$H_q(W)$ has dimension $|W|$. This implies that for $K=\Bbb C$ the algebra $K\otimes_{\Bbb Z}H(W)$ is a projective $K\otimes_{\Bbb Z}R$-module of rank $|W|$ (\cite{H}, Exercise 2.5.8(c)). Hence the same is true for any field $K$ of characteristic zero. 
But by Swan's theorem (\cite{La}, Corollary 4.10), any finitely generated projective module over the algebra of Laurent polynomials over a field is free.  Hence, the algebra $K\otimes_{\Bbb Z}H(W)$ is a free $K\otimes_{\Bbb Z}R$-module of rank $|W|$ 
(cf. also \cite{M1}, Proposition 2.5). This proves Theorem \ref{main}. 

\begin{corollary} Let $K=\Bbb Z[1/N]$ for $N\gg 0$. Then $K\otimes_{\Bbb Z}H(W)$ is a free $K\otimes_{\Bbb Z}R$-module of rank $|W|$. 
Hence the same holds when $K$ is a field of sufficiently large positive characteristic. 
\end{corollary} 

\begin{proof} 
Theorem \ref{weak} is valid (with the same proof) over any coefficient ring (see e.g. \cite{M1}, Theorem 2.14), i.e., for any $W$, the algebra $H(W)$ is module-finite over $R$. Hence by Grothendieck's  Generic Freeness Lemma (\cite{E}, Theorem 14.4), 
there exists an integer $L>0$ such that $H(W)[1/L]$ is a free $\Bbb Z[1/L]$-module. 

Now let $v_1,...,v_r$ be generators of $H(W)$ over $R$, 
and $e_i,...,e_{|W|}\in H(R)$ be elements defining a basis of $\Bbb Q\otimes_{\Bbb Z}H(W)$ over $\Bbb Q\otimes_{\Bbb Z}R$
(they exist by Theorem \ref{main}). Then $v_i=\sum_j a_{ij}e_j$ for some 
$a_{ij}\in \Bbb Q\otimes_{\Bbb Z}R$. So for some integer $D>0$ we have $Dv_i=\sum_j b_{ij}e_j$, with $b_{ij}\in R$. 
Since $H(W)[1/L]$ is a free $\Bbb Z[1/L]$-module, the same relation holds in $H(W)[1/L]$. Thus, for $N=LD$, 
$H(W)[1/N]$ is a free $R[1/N]$-module with basis $e_1,...,e_{|W|}$.  
\end{proof} 

\begin{remark} 1. The proof of Theorem \ref{main} does not extend to positive characteristic, since the proof of Theorem \ref{Los} uses complex analysis (the Riemann-Hilbert correspondence). 

2. The last step of the proof of Theorem \ref{main} (Swan's theorem) is really needed for purely aesthetic purposes, to establish the original formulation of the conjecture on the nose. As usual, for practical purposes it is normally sufficient to know only that the algebra $K\otimes_{\Bbb Z}H(W)$ is a {\it projective} $K\otimes_{\Bbb Z}R$-module. In fact, for most applications, including the ones mentioned in Remark \ref{appl}, already Losev's Theorem \ref{Los} is sufficient. 

3. One would like to have a stronger version of Theorem \ref{main}, giving a set-theoretical splitting $W\to B_W$ of the homomorphism $\pi$  
whose image is a basis of the Hecke algebra. For instance, when $W$ is a Coxeter group, then such a splitting is well known and is obtained by taking reduced expressions in the braid group. Such a version is currently available (over any base ring) for all irreducible complex reflection groups except $G_{17},...,G_{21}$, see \cite{M3}, \cite{Cha1,Cha2}. 

4. Here is an outline of the proof of Theorem \ref{Los} given in \cite{L}. Let
$q=e^{2\pi ic}$, and let $\bold H_c(W)$ be the rational Cherednik
algebra of $W$ with parameter $c$, \cite{GGOR}.  Let
$M\in {\mathcal O}_c(W)$ be a module from the category $\mathcal O$ for
this algebra. It is shown in \cite{GGOR} that the localization of $M$
to the set $\h^{\rm reg}$ of regular points of the reflection
representation $\h$ of $W$ is a vector bundle on $\h^{\rm reg}$ 
with a flat connection.
So for every $x\in \h^{\rm reg}$ we get a monodromy representation of
the braid group $\pi_1(\h^{\rm reg}/W)$ on the fiber $M_x$, which is shown in \cite{GGOR} to
factor through $H_q(W)$.  This representation is denoted by $KZ(M)$,
and the functor $M\mapsto KZ(M)$ is called the Knizhnik-Zamolodchikov
(KZ) functor. It is shown in \cite{GGOR} that the representation
$KZ(M)$ of $H_q(W)$ factors through a certain quotient $H_q'(W)$ of
$H_q(W)$ of dimension $|W|$.  Thus, Theorem \ref{Los} is equivalent to
the statement that every finite dimensional representation of $H_q(W)$
is of the form $KZ(M)$ for some $M$.

To show this, let $\h^{\rm sr}$ be the complement of the
intersections of pairs of distinct reflection hyperplanes in $\h$.
Take a finite dimensional representation $V$ of
$H_q(W)$, and let $N=N_V$ be the vector bundle with a flat connection with
regular singularities on $\h^{\rm reg}$ corresponding to $V$ under Deligne's
multidimensional Riemann-Hilbert correspondence. One then extends $N$
to a vector bundle $\widetilde N$ on $\h^{\rm sr}$ compatibly with the
$\bold H_c(W)$-action. One then defines
$M:=\Gamma(\h^{\rm sr},\widetilde N)$ and shows that
$M\in {\mathcal O}_c(W)$ and $KZ(M)=V$, as desired.

5. Here is an outline of the proof of Theorem \ref{weak} given in \cite{ER}. For the infinite series of complex reflection groups the result was proved in \cite{BMR}. 
Thus, let $W\subset GL_2(\Bbb C)$  be an exceptional complex reflection group of rank $2$, of type $G_4,...,G_{22}$. 
Then the intersection of $W$ with the scalars is a finite cyclic group generated by an element $Z$. This element defines a central element of $H_q(W)$, 
which we will also call $Z$. Let $W/\langle Z\rangle=G\subset PGL_2(\Bbb C)=SO_3(\Bbb C)$. Then $G$ is the group of even elements 
in a Coxeter group of type $A_3$, $B_3$, or $H_3$. Using the theory of length in these Coxeter groups, 
it is shown that $\Bbb C\otimes_{\Bbb Z}H(W)$ is generated by $|G|$ elements as a module over $\Bbb C\otimes_{\Bbb Z}R[Z,Z^{-1}]$.
Moreover, taking the determinant of the braid relation of this algebra in its finite dimensional representations, we find that 
$Z^d$ is an element of $\Bbb C\otimes_{\Bbb Z}R$ for some $d$. This implies that $\Bbb C\otimes_{\Bbb Z}H(W)$ is a finite rank module 
over $\Bbb C\otimes_{\Bbb Z}R$, as desired. 

We note that this argument  works over an arbitrary base ring. A much more detailed description of this argument 
is given in \cite{Cha2}.    
\end{remark}

{\bf Acknowledgements.} The author thanks I. Marin for many useful comments and references.


\begin{thebibliography}{9999} 
\bibitem[A]{A} S. Ariki,
Representation theory of a Hecke algebra of
$G(r;p;n)$, J. Algebra 177 (1995), 164-185.

\bibitem[AK]{AK}  S. Ariki, K. Koike,
A Hecke algebra of $(\Bbb Z/rZ)\wr S_n$ and construction of its irreducible representations, Adv. Math.
106, 216-243 (1994). 

\bibitem[BMR]{BMR} M. Brou\'e, G. Malle, R. Rouquier, Complex reflection groups, braid groups, Hecke algebras, J. reine angew. Math. 500 (1998), pp. 127---190. 

\bibitem[Cha1]{Cha1} E. Chavli, The BMR freeness conjecture for the tetrahedral and octahedral families
arXiv:1607.07023.

\bibitem[Cha2]{Cha2} E. Chavli, The Brou\'e-Malle-Rouquier conjecture for the exceptional groups of rank 2, Ph.D. Thesis, U. Paris Diderot, 2016. 
arXiv:1608.00834. 

\bibitem[DM]{DM} P. Deligne and G. D. Mostow, Commensurabilities among Lattices in PU(1,n), Annals of Mathematics Studies 132, Princeton Univ. Press, Princeton, 1993. 

\bibitem[E]{E} D. Eisenbud, Commutative algebra with a view towards algebraic geometry,
Graduate Texts in Mathematics, v.150, 1994. 

\bibitem[ER]{ER} P. Etingof, E. Rains, Central extensions of preprojective algebras, the quantum Heisenberg algebra, and 2-dimensional complex reflection groups, Journal of Algebra, Volume 299, Issue 2, Pages 570Ð--588. 

\bibitem[GGOR]{GGOR} 
V. Ginzburg, N. Guay, E. Opdam, R. Rouquier. On category O for rational Cherednik algebras. Invent. Math., 154 (3) (2003), pp. 617---651.

\bibitem[H]{H} R. Hartshorne, Algebraic geometry, Graduate texts in mathematics, Springer, 1977.

\bibitem[La]{La} T.Y.Lam, Serre's problem
on Projective modules, Springer monographs in mathematics, 2006. 

\bibitem[L]{L} I. Losev, Finite-dimensional quotients of Hecke algebras, Algebra and Number Theory, 
Vol. 9 (2015), No. 2, 493---502.  

\bibitem[M1]{M1}  I. Marin,
The freeness conjecture for Hecke algebras of complex reflection groups, and the case of
the Hessian group $G_{26}$, J. Pure and Applied Algebra, 218
(2014), 704-720.

\bibitem[M2]{M2}  I.  Marin, The  cubic  Hecke  algebra  on  at  most  5  strands,  J.  Pure  and Applied  Algebra
216 (2012), pp. 2754-2782.

\bibitem[M3]{M3} I. Marin, Report of the Brou\'e-Malle-Rouquier conjectures, September 13, 2015, 
http://www.lamfa.u-picardie.fr/marin/arts/reportBMR.pdf, to appear in proceedings of the INDAM intensive period ``Perspectives in Lie theory".

\bibitem[MP]{MP} I. Marin, G. Pfeiffer, The BMR freeness conjecture for the 2-reflection groups, arXiv:1411.4760,
to appear in ``Mathematics of Computation".  

\bibitem[Sh]{Sh} P. Shan,  Crystals of Fock spaces and cyclotomic rational double affine Hecke algebras
Annales scientifiques de l'\'Ecole Normale Sup\'erieure (2011)
Volume: 44, Issue: 1, page 147---182.

\bibitem[ST]{ST} Shephard, G. C.; Todd, J. A. Finite unitary reflection groups, Canadian Journal of Mathematics, v.6, pp. 274Ð304, 1954. 

\end{thebibliography}
\end{document}